\documentclass[11pt,reqno]{amsart}
\usepackage{setspace}

\usepackage{color}
\usepackage[pagebackref]{hyperref}
\usepackage{enumitem}
\usepackage{mathtools}
\usepackage{amsfonts}
\usepackage{amsmath}
\usepackage{amsthm}
\usepackage[dvipsnames]{xcolor}
\usepackage{graphicx}
\input{xy}
\xyoption{all}
\usepackage{mathtools}
\usepackage{tikz}
\usepackage{tikz-cd}
\usetikzlibrary{%
  matrix,%
  calc,%
  arrows%
}
\theoremstyle{plain}
\newtheorem{defi}{Definition}[section]
\newtheorem{thm}[defi]{ Theorem}

\newtheorem{prop}[defi]{Proposition}

\newtheorem{cor}[defi]{Corollary}
\newtheorem{lemma}[defi]{Lemma}
\newtheorem{nota}[defi]{Notations}
\newtheorem{rmk}[defi]{Remark}
\newtheorem{dis}[defi]{Discussion}
\newtheorem{ques}[defi]{Question}

\newcommand{\fm}{\mathfrak{m}}

\newcommand{\fp}{\mathfrak{p}}

\newcommand{\ra}{\rightarrow}

\DeclareMathOperator{\Hom}{Hom} \DeclareMathOperator{\Spec}{Spec}\DeclareMathOperator{\Image}{Im} \DeclareMathOperator{\coker}{Coker}
\DeclareMathOperator{\Ker}{Ker} \DeclareMathOperator{\Ann}{Ann} \DeclareMathOperator{\Supp}{Supp}\DeclareMathOperator{\Sym}{Sym} \DeclareMathOperator{\Ass}{Ass}\DeclareMathOperator{\Proj}{Proj}   
\DeclareMathOperator{\Ext}{Ext} \DeclareMathOperator{\Dim}{dim} 
    
\DeclareMathOperator{\grade}{grade}    \DeclareMathOperator{\G}{\boldsymbol{\gamma}}

\DeclareMathOperator{\Var}{V} 

\def\tt{{\bf t}}

\begin{document}
\title{A genus explanation of Buchsbaum-Rim multiplicity}
\author[ Vinicius Bou\c{c}a]{V. Bou\c{c}a}
\address{Vinicius Bou\c{c}a, Centro Federal de Educa\c{c}\~{a}o e  Tecnolgica Celso Suckow da Fonseca (CEFET-RJ), Av. Gov. Roberto Silveira, 1900, 28635-080, Nova Friburgo, Rio de Janeiro, Brazil, e-mail: vinicius.costa@cefet-rj.br }
\author[Thiago Fiel]{T.Fiel}
\address{Thiago Fiel da Costa Cabral, Departamento de Matem\'atica, Universidade Federal da Para\'iba, 58051-900 Jo\~ao Pessoa, PB, Brazil, email:	thiago.fiel@academico.ufpb.br,  thiagofieldacostacabral@gmail.com}
\author[S. Hamid Hassanzadeh]{S.H.Hassanzadeh}
\address{Seyed Hamid Hassanzadeh, Centro de Tecnologia - Bloco C, Sala ABC\\
 Cidade Universit\'{a}ria da Universidade Federal do Rio de Janeiro\\
21941-909  Rio de Janeiro, RJ, Brazil\\
  e-mail: hamid@im.ufrj.br }
\author[Jose Naeliton]{J.Naeliton}
\address{Jose Naeliton, Departamento de Matem\'atica, Universidade Federal da Para\'iba, 58051-900 Jo\~ao Pessoa, PB, Brazil, email:}
\maketitle
\footnotetext{Mathematics Subject Classification 2020 (MSC2020). Primary 13H15, 14C17, Secondary 32S15.}
\footnotetext{The second author was supported by PhD scholarship from CAPES-Brasil.}
\begin{abstract}
    We study the Buchsbaum-Rim multiplicity, ab inicio,  through a Koszul-\v{C}ech spectral sequence.   We show that  the  Buchsbaum--Rim multiplicity is the arithmetic genus (Euler characteristic) of Koszul homology sheaves on a projective space over the base scheme.  
\end{abstract}
\section{Introduction}
The Buchsbaum-Rim multiplicity is a generalization of the Hilbert-Samuel multiplicity. While the Hilbert-Samuel multiplicity is a classical numerical invariant to study isolated singularities, the  Buchsbaum--Rim multiplicity is a modern algebraic tool to study singularities of higher codimensions. The importance and the geometric significance of the Buchsbaum-Rim multiplicity are due to seminal works of  Gaffney \cite{G92, G93} in the study of (Whitney) equisingularities. Kleiman has as well investigated the geometric meaning of Buchsbaum-Rim multiplicity and developed many aspects of the theory, e. g. \cite{KT96},\cite{K99} and \cite{K17}.

Though not surprising, it is not trivial  that the Buchsbaum-Rim multiplicity has explanations as to the Hilbert-Samuel multiplicity. Several properties of Hilbert-Samuel multiplicity are extended to the Buchsbaum-Rim one, for example {\it characterization of reduction} \cite{BR2}, \cite{KT96} and \cite{SUV}; {\it relation between multiplicity and reduction number} \cite{BUV}; Lech's inequality \cite{NW}; Projection formula and associativity formula \cite{K17}.  See as well,\cite{J}, \cite{CLU08}, and  \cite{BR3}.

In this paper, we describe the Buchsbaum-Rim multiplicity as the Euler characteristic (arithmetic genus) of special sheaves on a projective space over the base scheme. Yet, another geometric nature of this multiplicity. This is  an extension of  Serre's formula of Hilbert-Samuel multiplicity in which the multiplicity is described as the Euler characteristic of the  Koszul complex,  c.f. \cite[Theorem 4.7.4 and notes on page 203]{BH}.

Now, we  recall the definition of the Buchsbaum--Rim multiplicity. Throughout, we keep the notations defined in subsection \ref{Notations} below. Assume, besides, that $R$ is a Noetherian ring and $M\otimes_RL$ is a finite-length $R$-module.  Consider  the subalgebra $R[\gamma_1,\ldots, \gamma_f]$  of $S$. Buchsbaum and Rim \cite{BR2} show that the function
\begin{equation}\label{BRpoly}
P_{\Phi}(\nu,L):=\ell_R((\frac{S}{R[\gamma_1,\ldots, \gamma_f]})_{\nu}\otimes_RL)
\end{equation}
is eventually a polynomial function of $\nu$. Moreover,  in the case where $R$ is a Noetherian local ring, it is of degree $\Dim(L)+g-1$. The product of 
 $(\Dim(L)+g-1)!$ and the leading  coefficient  of $P_{\Phi}(\nu,L)$ is called the Buchsbaum-Rim multiplicity of $M$ with respect to $L$ and it is denoted by $e_L(M)$, $e(\Phi|L)$, $br_L(\Image({\Phi}))$, or $br(M,L)$. In the case where $L=R$, one may ignore the index $L$ in these notations.\footnote{In the original definition \cite[Pages 213,214]{BR2}, $e_L(M)$ is defined to be $(\Dim(R)+g-1)!$ times the coefficient of the  term of degree $(\Dim(R)+g-1)$ in $P_{\Phi}(\nu,L)$ when $M$ is a finite length module over the local ring  $R$. Clearly,  the later is zero if $\Dim(L)<\Dim(R)$, nevertheless,  the original definition provides the additivity  of the multiplicity function as in the classical Hilbert-Samuel case.    }

In one of the most difficult results of \cite{BR2}, Buchsbaum and Rim show in \cite[Theorem 4.2]{BR2}  that the difference function of  $P_{\Phi}(\nu,L)$ is indeed the Euler-Poincar{\'e} characteristic of a family of complexes. The first member of this family of complexes is, nowadays, called the Buchsbaum-Rim complex. This complex and the Eagon-Northcott complex are parts of another family of complexes which is called the generalized Koszul complex, after Kirby \cite{Kir1}; or Buchsbaum-Eisenbud family of complexes, after \cite{BE} (see also \cite[Appendix 2.6]{Eis}).

In \cite{BHa}, the first and the third named authors of this paper show that the family of Buchsbaum-Eisenbud complexes consists of  strands of a Koszul-\v{C}ech spectral sequence. With this tool in hand, in this paper, we show that the Euler-Poincar{\'e} characteristic of these complexes, when exists, is the Buchsbaum-Rim multiplicity, on one hand; and on the other hand, it is the alternating sum of the arithmetic genera of some Koszul homology modules.

The structure of the paper is as follows. In Section 2, we recall the result of \cite{BHa} about the Koszul-\v{C}ech spectral sequence. We choose the notation $\mathfrak{B}_\bullet(\Phi,L,\nu)$ to denote the complexes derived from this spectral sequence. These are the same as the generalized Koszul complex of Kirby, $K(\Phi,L,\nu)$; and the Buchsbaum-Eisenbud complex, $\mathcal{C}^{\nu}_{\bullet}$. In section 3, we (re)prove some of the basic properties of $\mathfrak{B}_\bullet(\Phi,L,\nu)$. Among them are the grade-sensitivity, in Corollary \ref{Cgrade}, and the $\Phi$-normality, in Proposition \ref{PHfinitelength}. These results are among the most important properties of $\mathfrak{B}_\bullet(\Phi, L,\nu)$- to witness, we cite \cite{LR}. The proofs we present here are shorter and more comprehensive in comparison with the classical proofs. Nonetheless, the main reason to put these results in this section is their roles in the proof of the main Theorem \ref{TSerre}. The main section of the paper is section 4. 
In Theorem \ref{TSerre}, we determine the relation between the Euler-Poincare characteristic of $\mathfrak{B}_\bullet(\Phi,L,\nu)$ and the Hilbert polynomials of Koszul homologies of symmetric algebra. This theorem is a vast generalization of Serre's theorem on the relation between the Hilbert-Samuel multiplicity and the length of Koszul homologies.  More precisely:
\smallskip

{\it {\bf Theorem 4.12.~}  Let $R$ be a Noetherian local ring, $M\otimes_RL$  a finite length $R$-module and  $H_j:=H_j(\G,S\otimes_RL)$ the $j$-th Koszul homology module.  Then  for any integer $\nu$
$$
P_{H_0}(\nu)-P_{H_1}(\nu)+\cdots+(-1)^fP_{H_f}(\nu) =
\left\{
\begin{array}{ccll}
br(M,L),&\text{~if $\Phi$ is  a parameter matrix for $L$,}\\\\
0,&\text{~otherwise}
\end{array}
\right.
$$
}
where $P_{H_j}(\nu)$ is the Hilbert polynomial of $H_j$. Clearly, God is in the details!

To establish this theorem, one needs a generalization of the concept of the Grothendieck-Serre formula (for the Hilbert functions) for projective space over any affine scheme instead of projective space over a closed point. Finally,  in Corollary \ref{CbrXi}  we justify the title of the paper.

\medskip
\begin{nota}\label{Notations}
In the paper we keep the following notations: 
 $R$ is a commutative ring of finite Krull dimension and $S=R[T_1,\cdots,T_g]$ is a standard  positively graded polynomial extension of $R$. $X=\Proj(S):=\mathbb{P}^{g-1}_{\Spec(R)}$.  For  a finitely generated graded $S$-module $\mathcal{H}$,  by $\Dim_X(\Supp(\mathcal{H}))$ we mean the maximum integer $n$ such that there exists a chain  of  $n+1$ homogeneous prime ideals in $S$ none of which  contains $\tt=(T_1,\cdots, T_g)$ the irrelevant ideal of $S$.  $\Phi=(c_{ij})$ is a $g\times f$ matrix with entries in $R$ and $f\geq g$.   $M:=\coker(\Phi)$ and $L$ is a finitely generated $R$-module.  $K_\bullet(\G;S)$ is the Koszul complex of $\{\gamma_j=\sum_{i=1}^gc_{ij}T_i\}_{j=1}^f$ over $S$ and $C^\bullet_{{\tt}}$ is the (algebraic) \v{C}ech complex of $\{T_1,\dots,T_g\}$ over $S$. 
\end{nota}


\section{Koszul-Cech spectral construction}\label{Sn}

We start with  the third quadrant  double complex $E^{-\bullet,-\bullet}:=K_\bullet(\G;S)\otimes_SC^\bullet_{{\tt}}$ with $K_0(\G)\otimes_SC^0_{\tt}$ at the origin and the Koszul terms on the rows.

The terms of the second page of the horizontal spectral sequence of $E^{\bullet,\bullet}$ are given by
\begin{equation}\label{Ehor}
^2E^{-p,-q}_{hor}=H^q_{{\tt}}(H_p(\G,S)).
\end{equation}
The first page of the vertical spectral sequence of $E^{\bullet,\bullet}$ is given by
$$
^1E^{-\bullet,-q}_{ver}=\left\{
\begin{array}{cl}
 0 & , \ \ q\neq g \\
 0 \rightarrow H^g_{{\tt}}(K_f(\G;S)) \rightarrow \cdots \rightarrow H^g_{{\tt}}(K_0(\G;S))\rightarrow 0	& , \ \ q=g  
\end{array}
\right.,
$$
and it follows that $^2E_{ver}={^\infty}E_{ver}$. Since $H^g_{{\tt}}(\bullet)\cong \bullet\otimes_SH^g_{{\tt}}(S)$, $^1E^{-\bullet,-g}_{ver}=K_\bullet(\G,H^g_{{\tt}}(S))$, the Koszul complex of $\gamma_1,\dots,\gamma_f$ with coefficients in $H^g_{{\tt}}(S)$. Thus $H_n({\rm Tot}_\bullet E^{\bullet, \bullet})\cong H_{n+g}(\G,H^g_{{\tt}}(S))$. Furthermore, since  $S$ is  a standard graded ring the inverse polynomial structure $H^g_{{\tt}}(S)\cong R[T_1^{-1},\dots,T_g^{-1}]$ implies that ${\rm end}(H^g_{{\tt}}(S))=-g$. It follows that, for $ \nu \leq f-g$, $(^1E_{ver})_{\nu}$ (the first page of the vertical spectral of $E$ on degree $\nu$) is given by
$$
 0 \rightarrow K_f(\G,H^g_{{\tt}}(S))_{\nu} \rightarrow \cdots \rightarrow K_{g+\nu+1}(\G,H^g_{{\tt}}(S))_{\nu} \stackrel{\delta_\nu}{\rightarrow} K_{g+\nu}(\G,H^g_{{\tt}}(S))_{\nu} \rightarrow 0
$$
 wherein ${\rm coker}(\delta_{\nu})=H_{g+\nu}(\G,H^g_{{\tt}}(S))_{\nu}$. 
 For $\nu>f-g$, $(^1E_{ver})_{\nu}=0$. As well in the case, the strand of the $K_\bullet(\G;S)$ at degree $\nu$ is the complex 
$$0 \rightarrow K_\nu(\G,S)_{\nu} \stackrel{(\partial_\nu)_{\nu}}{\rightarrow} K_{\nu-1}(\G,S)_{\nu} \rightarrow \cdots \rightarrow K_0(\G,S)_{\nu} \rightarrow 0$$
wherein 
\begin{equation}\label{Ehv}
H_\nu(\G,S)_{\nu}={\rm ker}(\partial_\nu)_{\nu}.
\end{equation}
The convergence of spectral sequences says that there exists a filtration of ${\rm Tot}_\bullet (E^{\bullet, \bullet})$ such that, in degree $\nu$,
\begin{equation}\label{EF1}
\cdots \subset F_1 \subset F_0=({\rm Tot}_\nu E)_{\nu}=(^\infty E^{-g-\nu,-g}_{ver})_{\nu}={\rm coker}(\delta_{\nu})
\end{equation}
and
$$
\frac{{\rm coker}(\delta_\nu)}{F_1}\cong (^\infty E^{-\nu,0}_{hor})_{\nu}.
$$
Thence we define the map $\tau_\nu:K_{g+\nu}(\G,H^g_{{\tt}}(S))_{\nu}\rightarrow K_\nu(\G,S)_{\nu}$ by the composition
\begin{equation}\label{Etau}
\xymatrix{
K_{g+\nu}(\G,H^g_{{\tt}}(S))_{\nu}\ar @{->>}[r] & {\rm coker}(\delta_\nu)\ar @{->>}[r] & \frac{{\rm coker}(\delta_\nu)}{F_1} \ar[r]^{\cong} & (^\infty E^{-\nu,0}_{hor})_{\nu}\ar @{^{(}->}[ddlll] \\
 & & & \\
 (^2E^{-\nu,0}_{hor})_{\nu} \ar @{=}[r] & H^0_{{\tt}}(H_\nu(\G;S))_{\nu}\ar @{^{(}->}[r] & H_\nu(\G;S)_{\nu} \ar @{^{(}->}[r] & K_\nu(\G;S)_{\nu}
}
\end{equation}
where the epimorphisms and monomorphisms are all canonical.  In particular,
\begin{equation}\label{Eimtau}
\Image(\tau_{\nu})=(^\infty E^{-\nu,0}_{hor})_{\nu}
\end{equation}

Now, for each $0\leq\nu\leq f-g$, by splicing the complexes $K_\bullet(\G,H^g_{{\tt}}(S))_{\nu}$ and $K_\bullet(\G,S)_{\nu}$ via $\tau_{\nu}$ we define the complex $\mathfrak{B}_\bullet(\Phi,\nu)$ to be 
\begin{equation}\label{EB=}
\begin{tikzcd}
\\
K_{\nu}(\G)_{\nu} \ar[r,"(\partial_\nu)_{\nu}"] \arrow[d, phantom, ""{coordinate, name=Z}]
& K_{\nu-1}(\G)_{\nu} \arrow[r]
&  \cdots \arrow[r] K_{0}(\G)_{\nu}&0  \\ 
0\to K_f(\G,H^g_{{\tt}}(S))_{\nu}\arrow[r]\cdots \arrow[r]&  K_{g+\nu+1}(\G,H^g_{{\tt}}(S))_{\nu} \arrow[r,"\delta_\nu" ] 
& K_{g+\nu}(\G,H^g_{{\tt}}(S))_{\nu}  \arrow[ull, "\tau_\nu"  description,rounded corners,
to path={--([xshift=2ex]\tikztostart.east)
|-(Z) [near end]\tikztonodes
-| ([xshift=-2ex]\tikztotarget.west)
-- (\tikztotarget)}]
\\
\end{tikzcd}
\end{equation}


For $\nu>f-g$, $\mathfrak{B}_\bullet(\Phi,\nu)$ is  $K_\bullet(\G,S)_{\nu}$, i.e.  

\begin{equation}\label{EB>}
0 \rightarrow K_\nu(\G,S)_{\nu} \stackrel{(\partial_\nu)_{\nu}}{\rightarrow} K_{\nu-1}(\G,S)_{\nu} \rightarrow \cdots \rightarrow K_0(\G,S)_{\nu} \rightarrow 0,
\end{equation}and
 for $\nu<0$ it is $K_\bullet(\G,H^g_{{\tt}}(S))_{\nu}$, i.e. 
\begin{equation}\label{EB<}
 0 \rightarrow K_f(\G,H^g_{{\tt}}(S))_{\nu} \rightarrow \cdots \rightarrow K_{g+\nu+1}(\G,H^g_{{\tt}}(S))_{\nu} \stackrel{\delta_\nu}{\rightarrow} K_{g+\nu}(\G,H^g_{{\tt}}(S))_{\nu} \rightarrow 0.
\end{equation}
According to Bou\c ca and Hassanzadeh \cite[Theorem 3.3]{BHa}, these complexes $\mathfrak{B}_\bullet(\Phi,\nu)$ are the same as the Buchsbaum-Eisenbud complexes \cite[Appendix 2]{Eis}, as well known as generalized Koszul complexes by Kirby  \cite{Kir1}.  One can further take coefficients in an $R$-module $L$ and define  $$\mathfrak{B}_\bullet(\Phi,L,\nu):=\mathfrak{B}_\bullet(\Phi,\nu)\otimes_RL.$$   

\section{Basic properties}
The complex $\mathfrak{B}_\bullet(\Phi,0)$ is the Eagon-Northcott complex of the matrix $\Phi$. In this case, ${\rm im}(\tau_0)={\rm Fitt}_0(M)=I_g(\Phi)\subseteq R$. Proposition \ref{Ptophom} is one of the main properties of the Eagon-Northcott complex. Here, regarding the viewpoint of $\mathfrak{B}_\bullet(\Phi,L,\nu)$, we not only present a new proof of vanishing of the last homology, but also we determine the set of associated primes. 
We keep the notations  as in \ref{Notations}.
\begin{prop}\label{Ptophom}
Let $R$ be a Noetherian ring.  Then for any 
$0\leq \nu\leq f-g$,
$$\Ass(H_{f-g+1}(\mathfrak{B}_\bullet(\Phi,L,\nu)))=\Supp(M)\cap\Ass(L)=\Ass({\rm Hom}_R(R/I_g(\Phi),L)).$$
\end{prop} 
\begin{proof}
For $\nu=f-g$, $\mathfrak{B}_{f-g+1}(\Phi,L,\nu)=K_f(\G,H^g_{\tt}(S)\otimes_RL)_{f-g}\simeq L$. Since $\tau_{\nu}$ is given by multiplication into the generators of $I_g(\Phi)$ \cite[Theorem 3.6]{BHa}, we have $H_{f-g+1}(\mathfrak{B}_\bullet(\Phi,L,\nu)))={\rm Hom}_R(R/I_g(\Phi),L).$ We recall that the set of associate primes of the latter is $\Var(I_g(\Phi))\cap\Ass(L)$ according to the 
Bourbaki's result,\cite[Exercise 1.2.27]{BH},  here  $R$  must be Noetherian.

For $0\leq\nu\leq f-g-1$, we have $H_{f-g+1}(\mathfrak{B}_\bullet(\Phi,L,\nu))=H_f(\G,H^g_{{\tt}}(S\otimes_RL))_{\nu}$, i.e., the $\nu$-th component of Koszul homology of $\gamma_1,\dots,\gamma_f$ with coefficients in $H^g_{{\tt}}(S\otimes_RL)$. The self-duality of the Koszul complex yields
$$H_{f-g+1}(\mathfrak{B}_\bullet(\Phi,L,\nu))={\rm Hom}_S(S/\G,H^g_{{\tt}}(S\otimes_RL))(-f)_{\nu}.$$
From the perfect pairing given by the multiplication $S_{[i]}\otimes_RH^g_{{\tt}}(S\otimes_RL)_{[-i-g]}\rightarrow L$, we obtain the duality ${\rm Hom}_S(S/\G,H^g_{{\tt}}(S\otimes_RL))(-f)_{\nu}\cong {\rm Hom}_R((S/\G)_{[f-g-\nu]},L)$, see for example  \cite[Page 866]{Jou}. 
Regarding the presentation  $R^f\stackrel{\Phi}{\rightarrow}{R^g}\rightarrow M\rightarrow 0$, $S/(\G)\cong{\rm Sym}_RM$. Note that, $\Supp_R(M)=\Supp_R(\Sym_R^{i}M)$ for $i\geq 1$. Therefore,  for $0\leq\nu\leq f-g-1$
$$
\begin{array}{rcl}
 {\rm Ass}_R({\rm Hom}_R((S/\G)_{[f-g-\nu]},L)) & = & {\rm Ass}_R({\rm Hom}_R({\rm Sym}_R^{f-g-\nu}M,L)) \\
  & = & {\rm Supp}_R({\rm Sym}_R^{f-g-\nu}M)\cap{\rm Ass}_RL \\
 & = & {\rm Supp}_R(M)\cap{\rm Ass}_RL \\
 & = & V(R/{\rm Ann}M)\cap{\rm Ass}_RL \\
 & = & V(R/I_g(\Phi))\cap{\rm Ass}_RL \\
 & = & {\rm Ass}_R({\rm Hom}_R(R/I_g(\Phi),L)),
\end{array}
$$

\end{proof}

\begin{prop}\label{PHfinitelength}
 $\Supp_R(H_i(\mathfrak{B}_\bullet(\Phi,L,\nu)))\subseteq \Supp_R(M\otimes_R L)$ for any integers $i$ and $\nu$. 
\end{prop}
\begin{proof}
Before proving the assertion of the proposition, we study some particular properties of Koszul homologies involved therein. Let $P$ be a prime ideal with $P\notin \Supp_R(M)$,  then
$$\left(\frac{S}{\G}\right)_P\cong({\rm Sym}_RM)_P\cong{\rm Sym}_{R_P}M_P=R_P,$$
since $M_P=0$. 
Thus, the ideal generated by  $\G_P$ is the same as that of ${\tt}_P$; moreover  $H_\bullet(\G_P,S_P\otimes_{R_P}L_P)\cong  H_\bullet({\tt}_P,S_P\otimes_{R_P}L_P)\otimes_{S_P}\wedge^\bullet S_P^{f-g} $. Since  $(T_1,...,T_g)$ is an $S\otimes_RL-$ regular sequence, we have 
\begin{equation}\label{Lgt}
\begin{array}{rcl}
 H_i(\G,S\otimes_{R}L)_P& \cong & H_i(\G_P,S_P\otimes_{R_P}L_P)\\
 & \cong &\sum_{j=0}^iH_j({\tt}_P,S_P\otimes_{R_P}L_P)\otimes_{S_P}\wedge^{i-j}S_P^{f-g} \\
 & = &  L_P\otimes_{S_P}S_P^{\binom{f-g}{i}}(-i)\\
 & \cong & L_P^{\binom{f-g}{i}}(-i), 
\end{array}
\end{equation}
for all $i$. We need the following facts out of Equation (\ref{Lgt})
\begin{equation}\label{E2p}
\left\{
\begin{array}{ccl}
(H_i(\G,S\otimes_{R}L)_P)_{\nu}= 0 & \text {if }i\neq \nu, \\
(H_i(\G,S\otimes_{R}L)_P)_{\nu}= 0 & \text {if }i>f-g; \text{~~and} \\
H_i(\G,S\otimes_{R}L)_P&  \text{is~} {\tt}_P\text{-torsion.} 
\end{array}
\right.
\end{equation}
 Then since $H_i(\G,S\otimes_{R}L)_P$ is ${\tt}_P\text{-torsion}$, for all $i$, the horizontal spectral sequence $^2E^{-p,-q}_{hor}$ defined in (\ref{Ehor})  collapses at the first row.  Therefore, by the convergence of spectral sequences, it follows that for any integer $i$
\begin{equation}\label{ELp}
\begin{array}{rcl}
H_i(\G,H^g_{{\tt}}(S\otimes_RL))_P & \cong & H^0_{{\tt}_P}(H_{i-g}(\G_P,S_P\otimes_{R_P}L_P))  \\
  & \cong & H_{i-g}(\G_P,S_P\otimes_{R_P}L_P)\\
  & = & H_{i-g}(\G,S\otimes_{R}L)_P.
\end{array}
\end{equation}
We now return to the assertion of the proposition. Let $P$ be a prime ideal. 
If $P\notin \Supp_R(L)$, then obviously $H_i(\mathfrak{B}_\bullet(\Phi,L,\nu))_P = 0.$ If $P\notin \Supp_R(M)$, then we have all of the above mentioned properties. Now, consider three cases.

If $\nu\geq f-g+1$, then, according to (\ref{EB>}), $H_i(\mathfrak{B}_\bullet(\Phi,L,\nu)_P=(H_i(\boldsymbol{\gamma},S\otimes_RL)_P)_{\nu}$ for all $i$. Thence (\ref{E2p}) shows that $H_i(\mathfrak{B}_\bullet(\Phi,L,\nu)_P=0$  for all $i$.

If $\nu\leq -1$, then, according to (\ref{EB<}), $H_i(\mathfrak{B}_\bullet(\Phi,L,\nu)_P)=(H_{\nu+g+i}(\boldsymbol{\gamma},H^g_\tt(S\otimes_RL))_P)_{\nu}$. Thence (\ref{ELp}) shows that $H_i(\mathfrak{B}_\bullet(\Phi,L,\nu)_P\cong H_{i+v}(\boldsymbol{\gamma},S\otimes_RL)_P)_{\nu}$ which is trivially zero  for $0\leq i<-\nu$, and it is zero for $i\geq -\nu$ due to (\ref{E2p}).

If $0\leq\nu\leq f-g$. Then according to (\ref{EB=}),
\begin{equation}\label{BLp}
H_i(\mathfrak{B}_\bullet(\Phi,L,\nu)_P) = \left\{
\begin{array}{ccl}
(H_i(\G,S\otimes_RL)_P)_{\nu} & , & 0\leq i\leq \nu-1;\\
(H_{i-1+g}(\G,H^g_{{\tt}}(S\otimes_RL))_P)_{\nu} & , & \nu+2\leq i\leq f-g+1.
\end{array}
\right.
\end{equation}
Hence (\ref{E2p}) and (\ref{ELp}) yield $H_i(\mathfrak{B}_\bullet(\Phi,L,\nu)_P)=0$, for $i\neq \nu,\nu+1$.

The map $(\tau_\nu)_P\otimes_{R_P}1_{L_P}:\mathfrak{B}_{\nu+1}(\Phi,L,\nu)_P \rightarrow \mathfrak{B}_\nu(\Phi,L,\nu)_P$ comes from the composition
$$
\xymatrix{
(K_{g+\nu}(\G,H^g_{{\tt}}(S\otimes_{R}L))_P)_{\nu}\ar @{->>}[r] & {\rm coker}((\delta_\nu)_P\otimes_{R_P}1_{L_P})\ar @{->>}[r] & \frac{{\rm coker}((\delta_\nu)_P\otimes_{R_P}1_{L_P})}{F_1^{L_P}} \ar[ddll]^{\cong}  \\
 & &  \\
(^\infty E^{-\nu,0}_{hor}(L)_P)_{\nu}\ar @{^{(}->}[r] &
 (^2E^{-\nu,0}_{hor}(L)_P)_{\nu} \ar @{^{(}->}[r] & (K_\nu(\G,S\otimes_{R}L)_P)_{\nu}. &  
}
$$ 
where the module $F_1^{L_P}$ is given by the convergence of the spectral sequence of the double complex $E(L)_P=(K_\bullet(\G,S\otimes_{R}L)\otimes_{S}C^\bullet_{{\tt}})_P$ in degree $\nu$. 

There is a module $\mathfrak{F}_1^{L_P}\subseteq \mathfrak{B}_{\nu+1}(\Phi,L,\nu)_P$ such that $F_1^{L_P}=\mathfrak{F}_1^{L_P}/{\rm im}((\tau_\nu)_P\otimes_{R_P}1_{L_P})$, and thus ${\rm ker}((\tau_\nu)_P\otimes_{R_P}1_{L_P})=\mathfrak{F}_1^{L_P}$. It follows that $H_{\nu+1}(\mathfrak{B}_\bullet(\Phi,L,\nu)_P)=F_1^{L_P}$. Note that (\ref{ELp}) implies $F_1^{L_P}=0$, since ${\rm coker}((\delta_\nu)_P\otimes_{R_P}1_{L_P})=(H_{\nu+g}(\G,H^g_{{\tt}}(S\otimes_RL))_P)_{\nu}\cong(H_\nu(\G,S\otimes_RL)_P)_{\nu}$,
and therefore, ${\rm ker}((\tau_\nu)_P\otimes_{R_P}1_{L_P})={\rm im}((\delta_\nu)_P\otimes_{R_P}1_{L_P})$ and $H_{\nu+1}(\mathfrak{B}_\bullet(\Phi,L,\nu)_P)=0.$ Finally, 
$$
\begin{array}{rcl}
{\rm im}((\tau_\nu)_P\otimes_{R_P}1_{L_P}) & = & {\rm coker}((\delta_\nu)_P\otimes_{R_P}1_{L_P}) \\
 & \cong & (H_\nu(\G,S\otimes_RL)_P)_{\nu}\\
 & = & {\rm ker}((\partial_\nu)_P\otimes_{R_P}1_{L_P})_{\nu},
\end{array}
$$
which implies that $H_\nu(\mathfrak{B}_\bullet(\Phi,L,\nu)_P)=0$.
\end{proof}
Besides the main role of Proposition \ref{PHfinitelength} in Theorem \ref{TxiB}, we prove the grade-sensitivity as an application of this proposition for the sake of completeness.

\begin{cor}\label{Cgrade}
Suppose that $R$ is Noetherian, $M\otimes_RL\neq 0$ and  $0\leq\nu< f-g$. Then 
$$\grade_R(I_g(\Phi),L) = \min\{i; H_{f-g+1-i}(\mathfrak{B}_\bullet(\Phi,L,\nu))\neq 0\}.$$
Moreover, for $\frak{g}=\grade_R(I_g(\Phi),L)$, $$H_{f-g+1-\frak{g}}(\mathfrak{B}_\bullet(\Phi,L,\nu)) \simeq \Ext_R^{\frak{g}}(\Sym_R^{f-g-\nu}(M), L).$$ 
\end{cor}
\begin{proof}
We prove by using induction on $m:=\min\{i; H_{f-g+1-i}(\mathfrak{B}_\bullet(\Phi,L,\nu))\neq 0\}.$ 

Let $m=0$. As we see in the course of the proof of  Proposition \ref{Ptophom}, $$H_{f-g+1}(\mathfrak{B}_\bullet(\Phi,L,\nu))={\rm Hom}_S(S/\G,H^g_{{\tt}}(S\otimes_RL))(-f)_{\nu}\cong {\rm Hom}_R((S/\G)_{[f-g-\nu]},L).$$

The latter is isomorphic to 
$$ \Hom_R(\Sym_R^{f-g-\nu}(M),L).$$
We also notice that, according to Proposition \ref{Ptophom}, ${\rm Hom}_R(R/I_g(\Phi),L)\neq 0$; so that $$\grade_R(I_g(\Phi),L)=0.$$
Now, suppose that $m>0.$ According to Proposition \ref{Ptophom}, $\grade_R(I_g(\Phi),L)>0.$ By Proposition \ref{PHfinitelength}, there exist an integer $s$ and an  $L$-regular element $x\in I_g(\Phi)$ such that $x^sH_\bullet(\mathfrak{B}(\Phi,L,\nu))=x^sM= 0$.  Hence, the exact sequence $$0\to L\stackrel{x^s}{\rightarrow} L \to L/x^sL\to 0,$$
yields 
 the following short exact sequences  for any $i$,
 $$0\to H_{f-g+1-i}(\mathfrak{B}_{\bullet}(\Phi,L,\nu))\to H_{f-g+1-i}(\mathfrak{B}_{\bullet}(\Phi,L/x^sL,\nu))\to H_{f-g-i}(\mathfrak{B}_{\bullet}(\Phi,L,\nu))\to 0.$$
Considering the values $i<m$ the equality about grade follows.    For the second assertion, we set $i=m$ and apply the induction hypothesis. We have 

$$
\begin{array}{rcl}
H_{f-g+1-m}(\mathfrak{B}_\bullet(\Phi,L,\nu)) &\cong & H_{f-g-m}(\mathfrak{B}_\bullet(\Phi,L/x^sL,\nu)) \\
 & \cong & \Ext_R^{m-1}(\Sym_R^{f-g-\nu}(M),L/x^sL)\\
 & \cong & \Ext_R^m(\Sym_R^{f-g-\nu}(M),L).
\end{array}
$$ 

\end{proof}

\begin{rmk} We mention some points about Corollary \ref{Cgrade}:
\begin{itemize}
\item{The last isomorphism  in the proof  is the well-known Rees formula, see \cite[Lemma 1.2.4]{BH}. To use this formula, one needs $x^s\Sym_R^{f-g-\nu}(M)=0$; so that necessarily $\nu<f-g$; However, in the case where $\nu=f-g$, with essentially the same proof, one   shows that  $$H_{f-g+1-{\frak{g}}}(\mathfrak{B}_\bullet(\Phi,L,\nu)) \simeq \Ext_R^{\frak{g}}(H_0(\mathfrak{B}_\bullet(\Phi,\nu)),L)$$  for ${\frak{g}}=\grade_R(I_g(\Phi),L)$.}
\item{This Corollary  proves the Eagon's classical result that $$\grade_R(I_g(\Phi),L)\leq f-g+1;$$}
\item{This Corollary  shows that  $\mathfrak{B}_\bullet(\Phi,L,\nu))$ is acyclic  for all $0\leq \nu< f-g$, if and only if $\mathfrak{B}_\bullet(\Phi,L,\nu))$ is acyclic  for some $0\leq \nu< f-g$, if and only if $\grade_R(I_g(\Phi),L)= f-g+1$. A key property of the family. See also \cite[Theorem A2.10]{Eis}.}
\end{itemize}
\end{rmk}

\section{The Euler-Poincar\'e characteristic }
In \cite[Theorem 4.2]{BR2} the authors show that the Buchsbaum-Rim multiplicity is the Euler characteristic of  $\mathfrak{B}_{\bullet}(\Phi, L,0)$. The goal of this section is to use the convergence of spectral  sequence of $E^{\bullet,\bullet}$, defined in the first section,  to write the Euler characteristic of $\mathfrak{B}_{\bullet}(\Phi, L,\nu)$ in terms of Hilbert polynomials of Koszul homologies.

The key result is Theorem \ref{TxiB}. To prove this theorem, we need some preparatory lemmas. Some of these results are due to the necessity to  generalize the concept of the Grothendieck-Serre formula (for  Hilbert functions) for projective space over an affine scheme instead of projective space over a closed point.
\begin{defi}\label{Lxi}
Let $C_\bullet : 0\rightarrow C_k \rightarrow \cdots \rightarrow C_1 \stackrel{d_1}{\rightarrow} C_0 \rightarrow 0$ be a finite  complex  with finite length homology modules. 
The alternating sum $\sum_{i=0}^k(-1)^i\ell_R(H_i(C_\bullet))$ is called the Euler characteristic of complex $C_\bullet$ and it is denoted by $\chi(H_{\bullet}(C_{\bullet}))$.
\end{defi}

 
 In the sequel we will use the following notations:

\begin{defi}\label{Drho}  Assume that $\mathcal{H}$ is a finitely generated graded $S$-module such that the sheaf cohomology modules $H^i(X,{\tilde {\mathcal{H}}}(\nu))$
are  finite length $R$-modules for all $i$ and $\nu$. We define $$h^i_{\mathcal{H}}(\nu):=\ell_R(H^i(X,\tilde{\mathcal{H}}(\nu)))$$ and $$\rho_{\mathcal{H}}(\nu):=\sum_{i=0}^{g-1}(-1)^ih^i_{\mathcal{H}}(\nu).$$ 
\end{defi}
Using a usual technique for long exact sequence, the next Lemma follows.
\begin{lemma}\label{Lshort}  Let $\mathcal{H}_1$,$\mathcal{H}_2$ and $\mathcal{H}_3$ be three finitely generated graded $S$-module such that all of the sheaf cohomology modules $H^i(X,{\tilde {\mathcal{H}_j}}(\nu))$
are finite length $R$-modules for all $i$, $j$ and $\nu$. If $\mathcal{H}_1$,$\mathcal{H}_2$ and $\mathcal{H}_3$ suit into a short exact sequence $0\ra \mathcal{H}_1\ra \mathcal{H}_2\ra \mathcal{H}_3\ra 0$ then for any integer $\nu$, $$\rho_{\mathcal{H}_2}(\nu)=\rho_{\mathcal{H}_1}(\nu)+\rho_{\mathcal{H}_3}(\nu).$$
\end{lemma}
\begin{lemma}\label{Lrho} Let $R$ be a Noetherian ring and assume that $\mathcal{H}$ is a finitely generated graded $S$-module such that the sheaf cohomology modules $H^i(X,\tilde {\mathcal{H}}(\nu))$ are  finite length $R$-modules for all $i$ and $\nu$. Then the function $\rho_{\mathcal{H}}:\mathbb{N}\to \mathbb{Z}$, defined in Definition \ref{Drho}, is a polynomial function of degree $\Dim_X(\Supp(\mathcal{H}))$ with positive leading coefficient.

\end{lemma}
\begin{proof}
 We first treat the case where  $\Dim_X(\Supp(\mathcal{H}))=-1$, that is $\Supp_X(\mathcal{H})=\emptyset$. In this case, $h^i_{\mathcal{H}}(\nu)=0$ for all $i$ and $\nu$. Thus $\rho_{\mathcal{H}}$ is just the zero function: to see that $h^i_{\mathcal{H}}(\nu)=0$, we consider the ideal transform functor, $D_{\tt}(-)$, according to the notations  in \cite[Chapter 2]{BS}. With this setting, $H^i(X,\tilde{\mathcal{H}}(\nu))=\mathcal{R}^iD_{\tt}(\mathcal{H})_\nu$. Hence, if $\Dim_X(\Supp(\mathcal{H}))=-1$ then $\Ann_S(\mathcal{H})$, which is a homogeneous ideal of $S$, must contain a power of $\tt$. So that  $\Gamma_{\tt}(\mathcal{H})=\mathcal{H}$, that is $\mathcal{H}$ is a $\tt$-torsion $S$-module. Hence $D_{\tt}(\mathcal{H})=0$ and $H^i_{\tt}(\mathcal{H})=0$ for all $i\geq 2$ by \cite[Corollary 2.2.10]{BS}. 
 
 For the rest of the proof,  instead of stating a proof similar to the case where $R$ is an Artinian local ring, we show how one can reduce the problem to that case. 
 
 Notice that there is a chain $0\subseteq N_0\subseteq \cdots \subseteq N_e=\mathcal{H}$ of graded submodules of $\mathcal{H}$ such that for each $i$, $N_{i+1}/N_i\simeq S/\fp_i(a_i)$ where $\fp_i$ is a graded prime ideal of $S$. 
 Some of these prime ideals have the maximum possible dimension which is the Krull dimension of $\mathcal{H}$. So that if one shows that the proposition holds for modules of the form $S/\fp$, with $\fp$ a homogeneous prime ideal, the result follows from Lemma \ref{Lshort}.
 
 Now, notice that $D_{\tt}(\mathcal{H})=D_{\tt}(\mathcal{H}/\Gamma_{\tt}(\mathcal{H}))$ and $H^i_{\tt}(\mathcal{H})=H^i_{\tt}(\mathcal{H}/\Gamma_{\tt}(\mathcal{H}))$ for $i\geq 2$. Consequently, $\rho_{\mathcal{H}}=\rho(\mathcal{H}/\Gamma_{\tt}(\mathcal{H}))$.  Moreover $\Dim_X(\Supp(\mathcal{H}))=\Dim_X(\Supp(\mathcal{H}/\Gamma_{\tt}(\mathcal{H})))$. So that one may suppose that $\mathcal{H}$ is a $\tt$-torsion free $S$-module.
 
 Next, we show that $R$ can be reduced to an Artinian local ring, assuming $\mathcal{H}=S/\fp$ and $\fp\not\supseteq (\tt)$. Since $\Gamma_{\tt}(\mathcal{H})=0$, $\mathcal{H}$ is a graded $S$-submodule  of $D_{\tt}(\mathcal{H})=\oplus_\nu H^0(X,\tilde{\mathcal{H}}(\nu))$. By our hypothesis, for each integer $\nu$, $H^0(X,\tilde{\mathcal{H}}(\nu))$ is a finite length $R$-module. Hence $\mathcal{H}_\nu$ is a finite length $R$-module. Since $\mathcal{H}$ is a finitely generated $S$-module,  its generators are concentrated in a finite number of graded components of $\mathcal{H}$, say $\mathcal{H}_{i_1},\cdots, \mathcal{H}_{i_q}$. Any of $\mathcal{H}_{i_j}$ is an $R$-module of finite length, thus its support consists of a finite number of maximal ideals of $R$. Although $\mathcal{H}$ is not necessarily a finitely generated $R$-module, the $R$-support of $\mathcal{H}$ will be the union of these maximal ideals which is a finite set, say $\{\fm_1,\cdots,\fm_c\}$. 

Then $\mathcal{H}$ is annihilated by a power of  $(\fm_1\cdots\fm_c)$ say $(\fm_1\cdots\fm_c)^k$.  The change of base ring theorem for local cohomologies, \cite[Theorem 2.2.24]{BS},  shows that $h^i_{\mathcal{H}}(\nu):=\ell_R(H^i(X,\tilde{\mathcal{H}}(\nu)))=\ell_{R'}(H^i(X,\tilde{\mathcal{H}}(\nu)))$ where $R'=R/(\fm_1\cdots\fm_c)^k$. Moreover $\Supp_X(\mathcal{H})=\Supp_{X'}(\mathcal{H})$ where $X'=X\times_{\Spec(R)}\Spec(R')$. Hence we may  substitute $R$ with $R'$ which is an Artinian semi-local ring. 

Considering the decomposition series for the finite length $R'$-module  $H^i(X,\tilde{\mathcal{H}}(\nu))$, it is easy to see that $\ell_{R'}(H^i(X,\tilde{\mathcal{H}}(\nu)))=\sum_{j=1}^c\ell_{R'_{\fm_j}}(H^i(X,\tilde{\mathcal{H}}(\nu))_{\fm_j})$. Consequently, the proof of the assertion reduces to the case where $R$ is an Aritinian local ring. 

For Artinian local ring $R$, the proof of this theorem is indeed the classical proof of Grothendieck-Serre formula, see for example \cite[Theorem 4.1.3 and Theorem 4.4.3]{BH}. We notice that since the polynomial ring $S$ is assumed to be standard, the function $\rho_{\mathcal{H}}$ is indeed a polynomial, whereas it is a quasi-polynomial in the general case. 
\end{proof}

The next Lemma is an important case where the conditions of Lemma \ref{Lrho} hold. 
\begin{lemma}\label{LHfinitelength} Let $R$ be a Noetherian ring and suppose that $M\otimes_RL$ is a finite length $R$-module. For  $p=0,\cdots, f$, let $H_p=H_p(\G,S\otimes_RL)$ be the Koszul homology modules  with sheafification $\widetilde{H_p}$. Then   the $R$-modules 
$$H^q(X,\widetilde{H_p}(\nu))$$  
are finite length for all $q$ and $\nu$. 
\end{lemma}
\begin{proof} The terms on the second page of the horizontal spectral sequence of third quadrant  double complex $E^{\bullet,\bullet}=K_\bullet(\G;S)\otimes_SC^\bullet_{{\tt}}\otimes_RL$ in degree $\nu$ are 
$H^q_{{\tt}}(H_p(\G,S\otimes_RL))_{\nu}$,
for $0\leq p\leq f$ and $0\leq q\leq g$. (We refer to Section \ref{Sn}, for the required properties and notations related to this spectral sequence.)


First, notice that $H^q_{{\tt}}(H_p(\G,S\otimes_RL))_{\nu}$ is of finite length for $q\geq 1$.
 In fact, if $\mathfrak{P}\notin \Supp_R(M\otimes_RL)$, then either $M_{\mathfrak{P}}=0$ or $L_{\mathfrak{P}}=0$. The latter, clearly, implies that  $H^q_{{\tt}}(H_p(\G,S\otimes_RL))_{\mathfrak{P}}=0$. In the former case, the map $$\Phi_{\mathfrak{P}}: R_{\mathfrak{P}}^f\to R_{\mathfrak{P}}^g$$is surjective. Hence  the ideal generated by $\G$ is the same as the ideal  generated by ${\tt}$. This fact implies that  $H_p(\G,S\otimes_RL)_{\mathfrak{P}}$ is ${\tt}_{\mathfrak{P}}$-torsion for all ${\mathfrak{P}}$,  and thus, $H^q_{{\tt}}(H_p(\G,S\otimes_RL))_{\mathfrak{P}}=0$ for all $q\geq 1$.

Therefore in  any degree $\nu$, $H^q_{{\tt}}(H_p(\G,S\otimes_RL))_{\nu}$ is a finitely generated $R$-module whose support is contained in the support of $M\otimes_RL$. The latter consists of maximal ideals; so that 
\begin{equation}\label{EH^q}
H^q_{{\tt}}(H_p(\G,S\otimes_RL))_{\nu} \text{~~is of  finite length for any~~} q\geq 1.
\end{equation}
Notice that $H^q(X,\widetilde{H_p(\G,S\otimes_RL)}(\nu))=H^{q+1}_{{\tt}}(H_p(\G,S\otimes_RL))_{\nu}$ for $q\geq 1$.  

It remains to show that $D_{\tt}(H_p(\G,S\otimes_RL))_{\nu}$ is a finite length $R$-module for all $\nu$. We study three cases

Case 1. If $\nu\geq f-g+1$. In this case,   $H_p(\G,S\otimes_RL)_{\nu}=H_p(\mathfrak{B}_{\bullet}(\Phi,L,\nu))$ for all $p$, according to the structure of $\mathfrak{B}_{\bullet}(\Phi,L,\nu)$ which is explained in (\ref{EB>}).   Proposition \ref{PHfinitelength} then shows that these homology modules have finite length. 

Case 2. If $\nu\leq -1$.   $H_p(\G,S\otimes_RL)_{\nu}$ is a subquotient of  $\Lambda^p(S^f(-1)\otimes_RL)_{\nu}$ for all $p$. The latter is zero be degree discussion.

Case 3. If $0\leq\nu\leq f-g$, we consider three other cases

Case 3.1.   If $p>\nu$, then  $H_p(\G,S\otimes_RL)_{\nu}$ is a subquotient of  $\Lambda^p(S^f(-1)\otimes_RL)_{\nu}$ for all $p$. The latter is zero be degree discussion.

Case 3.2.    If $p< \nu$, then $ H_p(\G,S\otimes_RL)_{\nu}=H_p(\mathfrak{B}_{\bullet}(\Phi,L,\nu))$, according to (\ref{EB=}), which is of finite length by Proposition \ref{PHfinitelength}. 

In all of the above cases,  $D_{{\tt}}(H_\nu(\G,S\otimes_RL))_{\nu}$ is a finite length $R$-module by regarding the exact sequence in conjunction with (\ref{EH^q})
$$0\rightarrow H^0_{{\tt}}(H_p(\G,S\otimes_RL))_{\nu}\rightarrow H_p(\G,S\otimes_RL)_{\nu}\rightarrow D_{{\tt}}(H_p(\G,S\otimes_RL))_{\nu}\rightarrow H^1_{{\tt}}(H_p(\G,S\otimes_RL))_{\nu}\rightarrow 0.$$

Case 3.3.  $p=\nu$.  Based on the structure of $\mathfrak{B}_{\bullet}(\Phi,L,\nu)$ in (\ref{EB=}), $H_{\nu}(\mathfrak{B}_{\bullet}(\Phi,L,\nu))=\Ker(\partial_{\nu})_{\nu}/\Image(\tau_{\nu})$. According to (\ref{Ehv}), $\Ker(\partial_{\nu})_{\nu}=H_\nu(\G,S\otimes_RL)_{\nu}$. As well, $\Image(\tau_{\nu})= ~^\infty E^{-\nu,0}_{hor}$ by (\ref{Eimtau}). 

Therefore, we have 
\begin{equation}\label{EHvinfty}
\frac{H_\nu(\G,S\otimes_RL)_{\nu}}{{^\infty}E^{-\nu,0}_{hor}}=H_\nu(\mathfrak{B}_{\bullet}(\Phi,L,\nu))
\end{equation}

The latter is of  finite length, according to Proposition \ref{PHfinitelength}.

Finally, the finiteness of $D_{{\tt}}(H_\nu(\G,S\otimes_RL))_{\nu}$ follows from the exactness of the following  natural sequence 
$$0\rightarrow\frac{H^0_{{\tt}}(H_\nu(\G,S\otimes_RL))_{\nu}}{{^\infty}E^{-\nu,0}_{hor}}\rightarrow\frac{H_\nu(\G,S\otimes_RL)_{\nu}}{{^\infty}E^{-\nu,0}_{hor}}\rightarrow D_{{\tt}}(H_\nu(\G,S\otimes_RL))_{\nu}\rightarrow H^1_{{\tt}}(H_\nu(\G,S\otimes_RL))_{\nu}\rightarrow 0.$$

\end{proof}

The last technical lemma is the following  which is part of the folklore; so that we leave its proof.
\begin{lemma}\label{PElength}
Let ${^r}E\displaystyle{\Rightarrow} H$ be a convergent spectral sequence. Suppose that for some $r$, ${^r}E^{pq}$ is finite length for all $p,q$. Then for all $s\geq r$
$$\sum_n(-1)^n\ell(H_n)=\sum_{n}(-1)^n\left(\sum_{p+q=n}\ell({^s}E^{pq})\right)$$
\end{lemma}

We are now ready to present and prove the following main property of  $\chi(H_{\bullet}(\mathfrak{B}_{\bullet}(\Phi,L,\nu))$. 

\begin{thm}\label{TxiB}  Let $R$ be a Noetherian ring and suppose that $M\otimes_RL$ is a finite length $R$-module.   Let $\rho_j(\nu):=\rho_{_{H_j(\G,S\otimes_RL)}}(\nu)$  be the $\rho$ function defined in  Definition \ref{Drho} for $j$-th Koszul homology module $H_j(\G,S\otimes_RL)$. Then, for all integer $\nu$
$$\chi(H_{\bullet}(\mathfrak{B}_{\bullet}(\Phi,L,\nu)))=\sum_{j=0}^f(-1)^{j}\rho_j(\nu).$$
\end{thm}
\begin{proof}
The proof is a deep analysis   of the horizontal spectral sequence of $(E^{\bullet,\bullet})_{\nu}=(K_{\bullet}(\G,S\otimes_RL)\otimes_SC^\bullet_{\tt})_{\nu}$. 

Due to  Lemma \ref{LHfinitelength}, the modules $({^2}E^{-j,-q}_{hor})_{\nu}=H^q_{{\tt}}(H_j(\G,S\otimes_RL))_{\nu}$ have finite length for $q\geq 2$ and any  $j$. Furthermore, the epimorphism
$$D_{{\tt}}(H_j(\G,S\otimes_RL))_{\nu}\rightarrow H^1_{{\tt}}(H_j(\G,S\otimes_RL))_{\nu}\rightarrow 0,$$  implies that $H^1_{{\tt}}(H_j(\G,S\otimes_RL))_{\nu}$ has finite length for every $j$. 

We need to look into $H^0_{{\tt}}(H_j(\G,S\otimes_RL))_{\nu}$.

For $\nu\geq f-g+1$,
$H_j(\G,S\otimes_RL)_{\nu}=H_j(\mathfrak{B}_{\bullet}(\Phi,L,\nu))$ is of finite length  by the same reason as Case 1 in the proof of Lemma \ref{LHfinitelength}. So that  
$H^0_{{\tt}}(H_j(\G,S\otimes_RL))_{\nu}\subseteq H_j(\G,S\otimes_RL)_{\nu}$ is of finite length. 
For  $\nu\leq -1$, $H_j(\G,S\otimes_RL)_{\nu}$ is a subquotient of  $\Lambda^j(S^f(-1)\otimes_RL)_{\nu}$ for all $j$. The latter is zero by degree discussion, so that $H^0_{{\tt}}(H_j(\G,S\otimes_RL))_{\nu}=0$.
If $0\leq\nu\leq f-g$, since 
\begin{equation}\label{Epq}
H_j(\G,S\otimes_RL)_{\nu}=\left\{ 
\begin{array}{ccl}
H_j(\mathfrak{B}_{\bullet}(\Phi,L,\nu)) & , & j<\nu ;\\
0 & , & j>\nu
\end{array}
\right.,
\end{equation}
and $H^0_{{\tt}}(H_j(\G,S\otimes_RL))\subseteq H_j(\G,S\otimes_RL)$ Proposition \ref{PHfinitelength} yields that  $H^0_{{\tt}}(H_j(\G,S\otimes_RL))$ has finite length for $j\neq\nu$.

When $0\leq\nu\leq f-g$  the $R$-module $H^0_{{\tt}}(H_\nu(\G,S\otimes_RL))_{\nu}$ is not necessarily of finite length. However, as we see in  the proof of the Lemma \ref{LHfinitelength}(\ref{EHvinfty})
\begin{equation}\label{E00}
\frac{H_\nu(\G,S\otimes_RL)_{\nu}}{({^\infty}E^{-\nu,0}_{hor})_{\nu}}=H_\nu(\mathfrak{B}_{\bullet}(\Phi,L,\nu))
\end{equation}
is of finite length. 

Unless otherwise stated, suppose that $0\leq\nu\leq f-g$. So far we see that every terms in the spectral sequence $(E^{\bullet,\bullet})_{\nu}$ except $(E^{-\nu,0})_{\nu}$, is of finite length. In order  to relate the lengths of the homologies of  $\mathfrak{B}_{\bullet}(\Phi,L,\nu)$ to the length of the terms of $(^2E^{\bullet,\bullet}_{hor})_{\nu}$, we define a new spectral sequence $^rG^{\bullet,\bullet}$ which is equal to $(^rE_{hor}^{\bullet,\bullet})_{\nu}$ for $(-j,-q)\neq(-\nu,0)$ with the same differentials, and for $(-j,-q)=(-\nu,0)$ 
$$^rG^{-\nu,0}:=\frac{(^rE^{-\nu,0}_{hor})_{\nu}}{({^\infty}E^{-\nu,0}_{hor})_{\nu}}$$  
with the induced differentials. 
Hence the induced differentials are just the same maps as they were in  $(^rE^{\bullet,\bullet})_{\nu}$ however some parts of their kernels are already killed. 

The advantage of  $^rG^{\bullet,\bullet}$ is that its terms on the second page have finite length. We only need to notice that 
$$^2G^{-\nu,0}=\frac{H^0_{{\tt}}(H_\nu(\G,S\otimes_RL))_{\nu}}{(^{\infty}E^{-\nu,0}_{hor})_{\nu}}\subseteq\frac{H_\nu(\G,S\otimes_RL)_{\nu}}{({^\infty}E^{-\nu,0}_{hor})_{\nu}}=H_\nu(\mathfrak{B}_{\bullet}(\Phi,L,\nu))$$
 which is of finite length by Proposition \ref{PHfinitelength}. 
Since ${^\infty}G^{-\nu,0}=0$, the spectral sequence $^rG^{\bullet,\bullet}$ converges to $\mathfrak{H}_{\bullet}$, where 
$$
\mathfrak{H}_j=\left\{
\begin{array}{ccl}
H_{j+g}(\G,H^g_{{\tt}}(S\otimes_RL))_{\nu} & , & \nu+1\leq j\leq f-g\\
F_1 & , & j=\nu\\
0 & , & {\text otherwise.} 
\end{array}
\right.
$$
Here, $F_1$ is the  module defined in equation (\ref{EF1}) which is given by the convergence of $^rE_{hor}^{\bullet,\bullet}$. 
Notice that 
\begin{equation}\label{Echangei}
\mathfrak{H}_j=H_{j+1}(\mathfrak{B}_{\bullet}(\Phi,L,\nu))\text{~~for~~} \nu+1\leq j\leq f-g \text{~~and~~} \mathfrak{H}_\nu=F_1=H_{\nu+1}(\mathfrak{B}_{\bullet}(\Phi,L,\nu))
\end{equation}
in the same way as the proof of Proposition \ref{PHfinitelength}. Therefore Proposition \ref{PHfinitelength} implies that  all  terms  of $\mathfrak{H}_\bullet$ have finite length. 

Applying  Lemma \ref{PElength}, it follows that
$$\sum_{j=\nu}^{f-g}(-1)^j\ell(\mathfrak{H}_j)=\sum_j\sum_q(-1)^{j+q}\ell(^2G^{-j,-q}_{hor}).$$
 According to equation (\ref{Echangei}), we change the indices on the left side, 
$$-\sum_{j=\nu+1}^{f-g+1}(-1)^j\ell(H_{j}(\mathfrak{B}_{\bullet}(\Phi,L,\nu)))=$$
$$\sum_j\sum_{q\geq 2}(-1)^{j+q}\ell(H^q_{{\tt}}(H_j(\G,S\otimes_RL))_{\nu}) +$$
\begin{equation}\label{Esumsum}
    \sum_{j\neq \nu}(-1)^j\{\ell(H^0_{{\tt}}(H_j(\G,S\otimes_RL))_{\nu})-\ell(H^1_{{\tt}}(H_j(\G,S\otimes_RL))_{\nu})\}+
    \end{equation}
    \begin{equation}\label{Esumsum2}
(-1)^{\nu}\left\{l\left(\frac{H^0_{{\tt}}(H_\nu(\G,S\otimes_RL))_{\nu}}{({^\infty}E^{-\nu,0}_{hor})_{\nu}}\right)-\ell(H^1_{{\tt}}(H_\nu(\G,S\otimes_RL))_{\nu})\right\}.
\end{equation}
For $j\neq\nu$, we consider the following exact sequence where all terms have finite length
$$0\rightarrow H^0_{{\tt}}(H_j(\G,S\otimes_RL))_{\nu}\rightarrow H_j(\G,S\otimes_RL)_{\nu}\rightarrow D_{{\tt}}(H_j(\G,S\otimes_RL))_{\nu}\rightarrow H^1_{{\tt}}(H_j(\G,S\otimes_RL))_{\nu}\rightarrow 0$$
We have
\begin{equation}\label{E43}
\ell(H^0_{{\tt}}(H_j(\G,S\otimes_RL))_{\nu})-\ell(H^1_{{\tt}}(H_j(\G,S\otimes_RL))_{\nu})=\ell(H_j(\G,S\otimes_RL)_{\nu})-\ell(D_{{\tt}}(H_j(\G,S\otimes_RL))_{\nu}).
\end{equation}
For $j=\nu$, we consider the sequence 
$$0\rightarrow\frac{H^0_{{\tt}}(H_\nu(\G,S\otimes_RL))_{\nu}}{({^\infty}E^{-\nu,0}_{hor})_{\nu}}\rightarrow\frac{H_\nu(\G,S\otimes_RL)_{\nu}}{({^\infty}E^{-\nu,0}_{hor})_{\nu}}\rightarrow D_{{\tt}}(H_\nu(\G,S\otimes_RL))_{\nu}\rightarrow H^1_{{\tt}}(H_\nu(\G,S\otimes_RL))_{\nu}\rightarrow 0,$$
which is exact by a straightforward verification.  Thus
\begin{equation}\label{E44}
  \ell\left(\frac{H^0_{{\tt}}(H_\nu(\G,S\otimes_RL))_{\nu}}{({^\infty}E^{-\nu,0}_{hor})_{\nu}}\right)-\ell(H^1_{{\tt}}(H_\nu(\G,S\otimes_RL))_{\nu})
=\ell\left(\frac{H_\nu(\G,S\otimes_RL)_{\nu}}{({^\infty}E^{-\nu,0}_{hor})_{\nu}}\right)-\ell(D_{{\tt}}(H_\nu(\G,S\otimes_RL))_{\nu}).  
\end{equation}
 Now, plugging (\ref{E43}) and (\ref{E44}) in (\ref{Esumsum}) and (\ref{Esumsum2}), respectively, we have
$$-\sum_{j=\nu+1}^{f-g+1}(-1)^j\ell(H_{j}(\mathfrak{B}_{\bullet}(\Phi,L,\nu)))=$$
$$\sum_j\sum_{q\geq 2}(-1)^{j+q}\ell(H^q_{{\tt}}(H_j(\G,S\otimes_RL))_{\nu}) +$$
$$\sum_{j\neq \nu}(-1)^j\{\ell(H_j(\G,S\otimes_RL)_{\nu})-\ell(D_{{\tt}}(H_j(\G,S\otimes_RL))_{\nu})\}+$$
$$(-1)^{\nu}\left\{l\left(\frac{H_\nu(\G,S\otimes_RL)_{\nu}}{{^\infty}E^{-\nu,0}_{hor}}\right)-\ell(D_{{\tt}}(H_\nu(\G,S\otimes_RL))_{\nu})\right\}.$$
Plugging (\ref{Epq}) and (\ref{E00}) in the last two lines, we have 
$$-\sum_{j=\nu+1}^{f-g+1}(-1)^j\ell(H_{j}(\mathfrak{B}_{\bullet}(\Phi,L,\nu)))=$$
$$\sum_j\sum_{q\geq 2}(-1)^{j+q}\ell(H^q_{{\tt}}(H_j(\G,S\otimes_RL))_{\nu}) +$$
$$\sum_{j\neq \nu}(-1)^j(\ell(H_{j}(\mathfrak{B}_{\bullet}(\Phi,L,\nu)))-\ell(D_{{\tt}}(H_j(\G,S\otimes_RL))_{\nu}))+$$
$$(-1)^{\nu}(\ell(H_{\nu}(\mathfrak{B}_{\bullet}(\Phi,L,\nu)))-\ell(D_{{\tt}}(H_\nu(\G,S\otimes_RL))_{\nu})).$$

Finally, writing $H^q(X,\widetilde{H_j}(\G,S\otimes_RL)(\nu))=\mathcal{R}^qD_{\tt}(H_j(\G,S\otimes_RL))_{\nu}$, we obtain the equality

$$
\chi(H_{\bullet}(\mathfrak{B}_{\bullet}(\Phi,L,\nu))) =\sum_j(-1)^j\rho_j(\nu).
$$

For $\nu\geq f-g+1$ or $\nu\leq -1$, all the terms of the spectral sequence $({^r}E^{\bullet\bullet}_{hor})_{\nu}$ are all of finite length. Thus, without introducing the spectral sequence ${^r}G^{\bullet,\bullet}$, the computation of $\sum_{j=\nu}^{f-g}(-1)^j\ell(\mathfrak{H}_j)$ shows the asserted equality. 

\end{proof}

\begin{rmk} The proof of Theorem \ref{TxiB} would be essentially the computational parts in (\ref{Esumsum}) and  (\ref{Esumsum2}), if   $({^2}E^{-j,-q}_{hor})_{\nu}$ were of finite length for all $j$ and $q$. By the way,  this desire is true except for $j=\nu, q=0$. So that, we had to verify  the details thoroughly.  
\end{rmk}

Kirby in \cite{Kir2} investigated the Euler-Poincar{\'e} characteristics of complexes $\mathfrak{B}_{\bullet}(\Phi,L,\nu)$. \cite[Theorem 4]{Kir2} states that if $L$ is a Noetherian module then the value $\chi(H_{\bullet}(\mathfrak{B}_{\bullet}(\Phi,L,\nu)))$ is independent of $\nu$. 

Let  $(R,\fm)$ be a  Noetherian local ring and $M$ be  presented,  minimally, by the matrix $\Phi$. 
Recall that when $M\otimes_R L$ is an $R$-module of finite length, then the generalized Krull prime ideal theorem \cite[Corollary 3.6]{BR2} implies that $f\geq \Dim(L)+g-1$.
Now, consider the Buchsbaum-Rim polynomial $P_{\Phi}(\nu,L)$ defined in (\ref{BRpoly}). Set $d:=\Dim(L)+g-1$ and 
$$P_{\Phi}(\nu,L)=a_d\nu^d+a_{d-1}\nu^{\nu-1}+\ldots+a_0.$$

Then  
\cite[Theorem 4]{Kir2} in conjunction with \cite[Corollaries 4.3, 4.4, in the case where $p=1$]{BR2} imply that for any integer $\nu$
\begin{equation}\label{Exibr}
  \chi(H_{\bullet}(\mathfrak{B}_{\bullet}(\Phi,L,\nu)))=
\left\{
 \begin{array}{cccl}
  d!a_d,& \text{ if~~} f=\Dim(L)+g-1; \\
 0,& \text{ if~~} f>\Dim(L)+g-1.\\
 \end{array}
\right.  
\end{equation}
In other words, according to the definition in the introduction, 
\begin{equation}\label{Exibr}
  \chi(H_{\bullet}(\mathfrak{B}_{\bullet}(\Phi,L,\nu)))=
\left\{
 \begin{array}{cccl}
  br(M,L),& \text{ if~~} f=\dim(L)+g-1; \\
 0,& \text{ if~~} f>\dim(L)+g-1.\\
 \end{array}
\right.  
\end{equation}

 \begin{defi}\cite[Page 214]{BR2}
 Let $(R,\fm)$ be a  Noetherian local ring and $M$ and $R$-module minimally presented by a matrix $\Phi$ which  consists of  $g$ rows and $f$ columns.  Then $\Phi$ is called  a {\it parameter matrix for  $L$} if $M\otimes_R L$ is a finite length $R$-module  and  $f= \Dim(L)+g-1$.
 \end{defi}

 Citing the  Grothendieck-Serre formula \cite[Theorem 4.4.3]{BH}, one may wonder if the function $\rho$ defined in Definition \ref{Drho} is indeed the Hilbert polynomial. However, we recall that in  the graded ring $S=R[T_1,\ldots,T_g]$ the base ring $R$ is  a Noetherinan ring and not necessarily an Artinian ring. 

Yet, in what follows, we show that the function $\rho$ is indeed the Hilbert polynomial of  certain modules in the case where Theorem \ref{TxiB} applies.

Let $N$ be a finitely generated graded $S$-module such that $\ell_R(N_{\nu})$ is finite for all integer $\nu$.  In this case, an argument similar to that in the proof of Lemma \ref{Lrho} shows that $N$ is indeed a graded module over a ring with an Artinian base ring.  So that, one may  talk about Hilbert polynomial of $N$ in the classical sense. We denote this function by $P_N(\nu)$.  

\begin{lemma}\label{LPNsat} Let $R$ be a Noetherian rig and $N$  a finitely generated $S$-module whose graded components are of finite length. Set $N^{sat}=N/\Gamma_{\tt}(N)$. Then $$P_N(\nu)=P_{N^{sat}}(\nu)$$
for all $\nu$.
\end{lemma}
\begin{proof}
Since $N$ is a finitely generated $S$-module and each graded component of $N$ is a finite length $R$-module, an argument similar to that in the proof of Lemma \ref{Lrho} shows that there exist maximal ideals  $\{\fm_1,\cdots,\fm_c\}$ and an integer $k$ such that $N$ is a $R/(\fm_1\cdots \fm_c)^k$-module. Since, $\Gamma_{\tt}(N)$ is a finitely generated $S$-module, it is annihilated by a power of $\tt$. Thence it is annihilated by a product of maximal ideals $(\fm_i+\tt)$. Being a Noetherian $S$-module, the latter implies that it is an Artinian $S$-module. Therefore the following descending chain of graded $S$-submodules  of $\Gamma_{\tt}(N)$ stops
$$\Gamma_{\tt}(N)_{\geq 0}\supseteq \Gamma_{\tt}(N)_{\geq 1}\supseteq \cdots .$$
The degree argument then shows that $\Gamma_{\tt}(N)_{\nu}=0$ for all $\nu\gg 0$.
Now it  follows from the exactness of the  sequence $$0\ra \Gamma_{\tt}(N)\ra N\ra N^{sat}\ra 0$$
that 
$P_N(\nu)=P_{N^{sat}}(\nu)$
for all $\nu \gg 0$. However two polynomials coincide for infinite numbers if and only if they are the same
\end{proof}
\begin{lemma}\label{Lrho vs P}
Let $R$ be a Noetherian ring and assume that $\mathcal{H}$ is a finitely generated graded $S$-module such that the cohomology modules $H^i(X,\tilde {\mathcal{H}}(\nu))$ are  finite length $R$-modules for all $i$ and $\nu$. Then for all $\nu$ $$\rho_{\mathcal{H}}(\nu)=P_{\mathcal{H}^{sat}}(\nu)$$
 where $\mathcal{H}^{sat}=\mathcal{H}/\Gamma_{\tt}(\mathcal{H})$.
\end{lemma}
 The proof is due to  the Serre's vanishing theorems \cite[Theorem 16.1.5(ii) and  Corollary 16.1.6(iii)]{BS}.  We  notice that,  for all $\nu$, $\mathcal{H}^{sat}_{\nu}$ is of finite length, since it is a subset of $H^0(X,\tilde {\mathcal{H}}(\nu))$.  As well, for large enough $\nu$, $\rho_{\mathcal{H}^{sat}}(\nu)=\ell(\mathcal{H}^{sat}_{\nu})=P_{\mathcal{H}^{sat}}(\nu)$. The equality $\rho_{\mathcal{H}^{sat}}(\nu)=P_{\mathcal{H}^{sat}}(\nu)$
for all $\nu$ follows from the fact that $\rho_{\mathcal{H}^{sat}}(\nu)$ and $P_{\mathcal{H}^{sat}}(\nu)$ are both polynomials.

\begin{dis}
Going back to the discussion before Lemma \ref{LPNsat}, one may wish to find in the literature, a generalization of the theory of  Hilbert function for finitely generated $S$-module $N$ for which only eventual values of $\ell_R(N_{\nu})$ are finite. In the Scheme-theoretic point of view of projective varieties, this fact is  what researchers indeed deal with. However from the commutative algebra point of view, the issue is in the intervention of the  saturation part; as we did in Lemma \ref{Lrho vs P}. Thenceforth for a  finitely generated $S$-module $N$ such that $\ell_R(N_{\nu})$ is finite for all $\nu\geq \nu_0$, we use the notation of Hilbert polynomial $P_N(\nu):=P_{N_{\geq \nu_0}}(\nu)$.
\end{dis}

The next Theorem shows how the Buchsbaum-Rim multiplicity is expressed as the alternating sum of Hilbert polynomials.

\begin{thm}\label{TSerre}
Let $R$ be Noetherian, $M\otimes_RL$  a finite length $R$-module,  $H_j:=H_j(\G,S\otimes_RL)$  the $j$-th Koszul homology module, $\Gamma_{\tt}(H_j)$ the $\tt$-torsion part of $H_j$, and $H_j^{sat}:=H_j/\Gamma_{\tt}(H_j)$.  Then for any integer  $\nu$,
$$\chi(H_{\bullet}(\mathfrak{B}_{\bullet}(\Phi,L)))=P_{H_0^{sat}}(\nu)-P_{H_1^{sat}}(\nu)+\cdots+(-1)^fP_{H_f^{sat}}(\nu).$$
In particular, if $R$ is a  local ring and  $\Phi$ is a parameter matrix for $L$ then
$$br(M,L)=P_{H_0^{sat}}(\nu)-P_{H_1^{sat}}(\nu)+\cdots+(-1)^fP_{H_f^{sat}}(\nu);$$
 If $\Phi$ is not a parameter matrix for $L$ then 
$$0=P_{H_0^{sat}}(\nu)-P_{H_1^{sat}}(\nu)+\cdots+(-1)^fP_{H_f^{sat}}(\nu).$$
\end{thm}
\begin{proof} The proof follows by combining Theorem \ref{TxiB}, Lemma \ref{Lrho vs P} and (\ref{Exibr}).

\end{proof}
To show the importance of Theorem \ref{TSerre}, we mention how  this theorem generalizes the Serre's celebrated theorem about the Hilbert-Samuel multiplicity, e.g. \cite[4.7.10]{BH}.

\begin{cor}\label{Cserre}{\rm(Serre)} Let $R$ be a Noetherian local ring, 
$I=(a_1,\ldots,a_f)$  an ideal of definition of $R$ and $H_j:=H_j(a_1,\ldots,a_f;R)$  the $j$-th homology of the Koszul complex of $a_1,\ldots,a_f$. 
If $(a_1,\ldots,a_f)$ is  a system of parameters  then 
$$e(I,R)=\ell({H_0})-\ell({H_1})+\cdots+(-1)^f\ell({H_f});$$
If $(a_1,\ldots,a_f)$ is  not a system of parameters,
$$0=\ell_R({H_0})-\ell_R({H_1})+\cdots+(-1)^f\ell_R({H_f}).$$

\end{cor}
\begin{proof}
In Theorem \ref{TSerre} we set $g=1$, $M=R/I$ and  $\Phi=(a_1,\ldots,a_f)$. For any $\nu$, the complex $\mathfrak{B}_{\bullet}(\Phi,R, \nu)$  is isomorphic to the Koszul complex $K_{\bullet}(a_1,\ldots,a_f;R)$. $H_j(\G,S)_{\nu}\cong H_j(a_1,\ldots,a_f;R)$ and  $\Dim_X(\Supp(H_j(\G,S))=0$ for all $j$. Hence $P_{H_j^{sat}}(\nu)=\ell(H_j)$ for all $j$ and $\nu$.  We, as well, notice that $br(M)=e(I,R)$ in this case \cite{BR2}.
\end{proof}

We define the Euler characteristic \footnote{This definition is the same as  \cite[Definition 33.32.1]{S} in which $R$ is a field.} of a coherent sheaf $\mathcal{F}$ on $X$ relative to the scheme $Y=\Spec(R)$ to be the following   integer (in the case it is finite)
\begin{equation}\label{genus}
   \chi(X,\mathcal{F})=\sum_{j=0}(-1)^j\ell_R(H^j(X,\mathcal{F})). 
\end{equation}

In spacial cases $\chi(X,\mathcal{F})$ relates the degree of $\mathcal{F}$ with the genus of $X$. 

Our last result is a genus explanation of the  Buchsbaum-Rim multiplicity. 
\begin{cor}\label{CbrXi}  Keeping the same notations as in \ref{Notations}. Let $R$ be a Noetheiran local ring, $M$ a finite length $R$-module and  $H_j:=H_j(\G,S)$ the $j$-th Koszul homology module with sheafification $\widetilde{H_j}$.
If $\Phi$ is   a parameter matrix for $R$ then
$$br(M)=\sum_{j=0}^f(-1)^j\chi(X,\widetilde{H_j});$$
If $\Phi$ is not a  parameter matrix for $R$ then 
$$0=\sum_{j=0}^f(-1)^j\chi(X,\widetilde{H_j}).$$
\end{cor}
\begin{proof}
In Theorem \ref{TxiB}, put $\nu=0$ and use the definition (\ref{genus}).

\end{proof}
We close the paper with a question
\begin{ques}
Are the partial sums $$\chi^j(H_{\bullet}(\mathfrak{B}_{\bullet}(\Phi,L,\nu)))=P_{H_j^{sat}}(\nu)-P_{H_{j+1}^{sat}}(\nu)+\cdots+(-1)^fP_{H_f^{sat}}(\nu)$$
positive, for any integer $\nu$?

\end{ques}

\end{document}